\documentclass[10pt,oneside,reqno]{amsart}

\usepackage[a4paper,left=30mm,right=30mm,top=25mm,bottom=30mm]{geometry}
\usepackage{amsmath,amsthm,amssymb,mathtools}
\usepackage{longtable}
\usepackage{lmodern}
\usepackage[colorlinks=true,linkcolor=blue,citecolor=green,urlcolor=blue]{hyperref}

\numberwithin{equation}{section}

\newtheorem{theorem}{Theorem}[section]
\newtheorem{proposition}[theorem]{Proposition}
\newtheorem{lemma}[theorem]{Lemma}
\newtheorem{corollary}[theorem]{Corollary}
\theoremstyle{definition}
\newtheorem{definition}[theorem]{Definition}
\newtheorem{example}[theorem]{Example}
\newtheorem{remark}[theorem]{Remark}

\newcommand{\Z}{\mathbb Z}
\newcommand{\Q}{\mathbb Q}
\newcommand{\C}{\mathbb C}
\newcommand{\PP}{\mathbb P}
\newcommand{\Cl}{\operatorname{Cl}}
\newcommand{\Pic}{\operatorname{Pic}}
\newcommand{\Aut}{\operatorname{Aut}}
\newcommand{\Cr}{\operatorname{Cr}}
\newcommand{\Sing}{\operatorname{Sing}}
\newcommand{\Bs}{\operatorname{Bs}}
\newcommand{\rk}{\operatorname{rk}}

\title{Actions of $(\Z/4)^4$ on rationally connected threefolds}
\author{Konstantin Loginov}

\date{}

\begin{document}

\begin{abstract}
Let $G=(\Z/4)^4$. We prove that if $X$ is a rationally connected
threefold with a faithful action of $G$, then $X$ is
$G$-birational to the Fermat quartic threefold.  If $X$ is a terminal
$G\Q$-Fano threefold, this birational equivalence is biregular.
Consequently, the group $G$ acts faithfully on a rationally
connected threefold but does not embed into $\Cr_3(\C)$.  Combined
with earlier results, this yields a complete classification of the
pairs $(m,r)$ for which $(\Z/m)^r$ embeds into $\Cr_3(\C)$, and of
those for which it embeds into $\operatorname{Bir}(X)$ for a
rationally connected threefold~$X$.
\end{abstract}

\let\oldemptyset\emptyset
\let\emptyset\varnothing

\renewcommand{\familydefault}{\rmdefault}
\newcommand{\Addresses}{{
  \bigskip
  \footnotesize

  \
    \

\textsc{Steklov Mathematical Institute of the Russian Academy of
Sciences, Moscow, Russia, \\Centre of Pure Mathematics, MIPT, Moscow, Russia}
 \\
  \textit{E-mail:} \texttt{loginov@mi-ras.ru}
}}

\maketitle

\section{Introduction}

We work over the field of complex numbers $\C$.  The $n$-dimensional Cremona group
$\Cr_n(\C)$ is the group of birational automorphisms of $\PP^n$.
Finite subgroups of $\Cr_2(\C)$ were classified by
Dolgachev and Iskovskikh \cite{DolgachevIskovskikh}.  For $n=3$, a
complete classification is not known, although some
classes of finite groups are understood: see 
\cite{ProkhorovSimple} for simple groups.

From dimension three onward, Mori theory naturally leads one to
consider the larger class of rationally connected varieties.  A
faithful action on a rationally connected $n$-fold need not yield a
subgroup of $\Cr_n(\C)$, because the variety need not be rational.
Among non-abelian finite simple groups,
Prokhorov's classification shows that $\operatorname{PSL}_2(\mathbb
F_{11})$ is the only group that acts faithfully on a rationally
connected threefold but does not embed into $\Cr_3(\C)$. It acts on the
Klein cubic threefold
\cite[Theorems 1.3, 1.5 and Example 2.6]{ProkhorovSimple}.  A nonsimple
example is provided by $\mathfrak S_7$, which acts
faithfully on a smooth
Fano complete intersection of a quadric and a cubic in $\PP^5$, whereas
$\mathfrak S_7$ does not embed into $\Cr_3(\C)$, see
\cite{BeauvilleSextic} and
\cite[Proposition 1.1]{ProkhorovSymmetric}.  In this paper, we give an
abelian example, namely $G=(\Z/4)^4$.

Finite abelian subgroups of $\Cr_2(\C)$ were classified by Blanc
\cite{Blanc}.  The study of the three-dimensional case was initiated in
\cite{LoginovProduct}.  Recall the two classes of groups introduced
there.  A finite abelian group $A$ is said to be of \emph{product type}
if $A\simeq A_1\times A_2$, where $A_i\subset\Cr_i(\C)$ for $i=1,2$. Groups of product type act faithfully on rational threefolds of the
form $\mathbb{P}^1\times S$, where $S$ is a rational surface, and their
classification follows from \cite{Blanc}.

A finite abelian group is said to be of \emph{K3 type} if there is an exact sequence
\[
 1\longrightarrow C\longrightarrow A\longrightarrow H
 \longrightarrow1,
\]
where $C$ is cyclic and $H$ acts faithfully on a K3 surface.  
The following result gives the basic trichotomy.

\begin{theorem}[{\cite[Theorem 1.7]{LoginovProduct}}]
\label{thm-three-types}
Let $X$ be a rationally connected threefold and let
$A\subset\operatorname{Bir}(X)$ be a finite abelian group.  Then at
least one of the following holds:
\begin{enumerate}
\item $A$ is of product type,
\item $A$ is of K3 type,
\item $A$ acts faithfully on a terminal $A\Q$-Fano threefold $X'$ with
$|-K_{X'}|=\varnothing$ such that $X'$ is $A$-birational to $X$.
\end{enumerate}
In case {\rm(3)}, moreover, every $A\Q$-Mori fiber space
$A$-birational to $X$ is an $A\Q$-Fano threefold with empty
anticanonical system.
\end{theorem}

The three cases in Theorem \ref{thm-three-types} are not mutually exclusive.  
Groups
of K3 type which act on rationally connected threefolds were studied in \cite{LPZ}. 
It is conjectured
that the third case in Theorem \ref{thm-three-types} does not produce any further groups: every such
group should be of product type or of K3 type
\cite[Conjecture 1.8]{LoginovProduct}.  
The following result was proved in \cite{LPZ}.

\begin{theorem}[{\cite[Theorem 1.5]{LPZ}}]
\label{thm-intro-2}
Let $A$ be a finite abelian group acting faithfully on a rationally
connected threefold.  Then either $A$ is of product type, or the third
alternative of Theorem~\ref{thm-three-types} holds, or $A$ is
isomorphic to one of the groups
\begin{equation}
\label{eq-exceptional-groups}
 (\Z/4)^4,\qquad
 (\Z/6)^3\times\Z/2,\qquad
 (\Z/6)^2\times(\Z/3)^2,\qquad
 (\Z/8)^2\times\Z/4\times\Z/2.
\end{equation}
\end{theorem}

The four groups \eqref{eq-exceptional-groups} are of K3 type and not of product type, see
\cite[Table 1]{LoginovProduct}.  
Each of these four groups admits a faithful action on a non-rational,
rationally connected
threefold realized as a hypersurface in a weighted projective space,
see \cite[Example~1.8(1)--(4)]{LPZ}.
We expect
that none of the four groups in \eqref{eq-exceptional-groups} embeds
into $\Cr_3(\C)$. In this paper we prove this for the group $G=(\mathbb{Z}/4)^4$.  

Consider the diagonal action of $G=(\mathbb{Z}/4)^4$ on the Fermat quartic
\begin{equation}
\label{eq-fermat}
X_4=\{x_0^4+x_1^4+x_2^4+x_3^4+x_4^4=0\}\subset\PP^4,
\end{equation}
where, after choosing generators $g_1,\ldots,g_4$ of $G$, the element
$g_i$ multiplies $x_i$ by a primitive fourth root of unity and fixes
the remaining coordinates. We call this action of $G$ on $X_4$ {standard}. 
Our main result is as follows. 

\begin{theorem}
\label{thm-classification}
Let $X$ be a rationally connected threefold and let
$G=(\Z/4)^4\subset\operatorname{Bir}(X)$.  Then there exist an
automorphism $\alpha\in\operatorname{Aut}(G)$ and a birational map
$\varphi\colon X\dashrightarrow X_4$ such that
\[
 \varphi\circ g=\alpha(g)\circ\varphi
 \qquad\text{for every }g\in G,
\]
where $G$ acts on $X_4$ by the standard diagonal action
\eqref{eq-fermat}.
Moreover, if $X$ is a terminal $G\Q$-Fano threefold, then $\varphi$ is an isomorphism.
\end{theorem}

Since the Fermat quartic threefold is non-rational, we deduce the
following.

\begin{corollary}
\label{cor-cremona}
There is no embedding
$
 (\Z/4)^4\hookrightarrow\Cr_3(\C).
$
\end{corollary}

Each group in \eqref{eq-exceptional-groups} is a cyclic extension of a
maximal finite abelian group $H$ acting faithfully on a K3 surface
\cite{BH}.  For the groups  in
\eqref{eq-exceptional-groups}, the corresponding groups $H$ are
$(\Z/4)^3$, $(\Z/6)^2\times\Z/2$, $\Z/6\times(\Z/3)^2$, and
$\Z/8\times\Z/4\times\Z/2$, respectively.  Let
$H_{\mathrm{symp}}\subset H$ be the subgroup of automorphisms preserving
a nonzero holomorphic $2$-form on the corresponding K3 surface.  In the
same order, the groups $H_{\mathrm{symp}}$ are $(\Z/4)^2$,
$\Z/6\times\Z/2$, $(\Z/3)^2$, and $\Z/4\times\Z/2$, respectively.
These groups occur in Nikulin's classification of finite
abelian groups of symplectic automorphisms of K3 surfaces
\cite{NikulinK3}.

It is known that $X_4$ is birationally superrigid
\cite{IskovskikhManin}, and hence
$\operatorname{Bir}(X_4)=\operatorname{Aut}(X_4)$.  The standard
$G$-action on $X_4$ extends to a faithful action of the group
\begin{equation}
 \Gamma=(\mu_4^5/\mu_4)\rtimes\mathfrak S_5
 \simeq(\Z/4)^4\rtimes\mathfrak S_5
 \label{eq-full-fermat-group}
\end{equation}
where $\mathfrak S_5$ acts by permuting coordinates. 
Thus $\Gamma\subset\operatorname{Aut}(X_4)$ and
$|\Gamma|=4^4\cdot5!=30\,720$.  This is the maximal possible order of
the automorphism group of a smooth quartic threefold, attained only by
$X_4$ \cite[Theorem 1.1]{YangYuZhu}.  Consequently,
$\operatorname{Bir}(X_4)=\operatorname{Aut}(X_4)=\Gamma$.  Theorem
\ref{thm-classification} now gives the following corollary.

\begin{corollary}
\label{cor-unique-gamma-model}
Let $X$ be a rationally connected threefold and let
$\Gamma=(\Z/4)^4\rtimes\mathfrak S_5\subset\operatorname{Bir}(X)$.  Then
there exist an automorphism $\alpha\in\operatorname{Aut}(\Gamma)$ and
a birational map $\varphi\colon X\dashrightarrow X_4$ such that
$\varphi\circ g=\alpha(g)\circ\varphi$ for every
$g\in\Gamma$, where $\Gamma$ acts naturally on $X_4$.
\end{corollary}

In dimension two, Oguiso characterized the Fermat quartic K3 surface
as the unique K3 surface admitting a faithful action of
$(\Z/4)^3\rtimes\mathfrak S_4$
\cite[Theorem 1.2]{OguisoFermatK3}.

For a prime number $p$ and an integer $a\geq1$, one may ask how large
$r$ can be if $(\mathbb{Z}/p^a)^r\subset\operatorname{Bir}(X)$ for
some rationally connected threefold $X$.  More generally, for rationally
connected varieties $X$ of a fixed dimension, one seeks bounds for the
number of generators of finite abelian $p$-subgroups of
$\operatorname{Bir}(X)$.  Such bounds in dimension three were
established in a series of works
\cite{ProkhorovElementary,Prokhorov2Elementary,
ProkhorovShramovP,Kuznetsova3,XuRank,Loginov3}.  The corresponding
result in higher dimensions was obtained by Koll\'ar and Zhuang.

For a finite abelian group $A$, let $\mathfrak r(A)$ denote its minimal
number of generators.

\begin{theorem}[{\cite[Proposition~8, Remark~10, and
Corollary~11]{KZh24}}]
\label{thm-kollar-zhuang}
Let $p$ be a prime number and let $P$ be a finite abelian $p$-group
acting faithfully on a smooth projective rationally connected variety
$X$ of dimension $n$.  Then
\begin{equation}
\label{eq-bound}
 \mathfrak r(P)\leq\frac{pn}{p-1}.
\end{equation}
Moreover, there is a subgroup $H\subset P$ which has a fixed point on
$X$ and satisfies
\[
 \log_p[P:H]\leq\frac{n}{p-1}.
\]
There is also a decomposition
$
 P\simeq P_1\times P_2,
$
where $P_1$ can be generated by at most $n$ elements and
\[
 |P_2|\leq p^{n/(p-1)}.
\]
\end{theorem}

In particular, for $n=3$, inequality \eqref{eq-bound} gives
$\mathfrak r(P)\leq6$ for $p=2$, $\mathfrak r(P)\leq4$ for $p=3$, and
$\mathfrak r(P)\leq3$ for $p\geq5$.  These bounds are sharp and are
attained on rational varieties, see Example~\ref{ex-standard-actions}. 
Together with Theorem~\ref{thm-classification}, these estimates give
sharp bounds for powers of cyclic groups.

\begin{theorem}
\label{thm-sharp-bounds}
Let $G=(\Z/m)^r$, where $m\geq2$ and $r\geq1$.  Then
$G\subset\Cr_3(\C)$ if and only if one of the following holds:
\[
 m=2,\ r\leq6,\quad\quad
 m=3,\ r\leq4,\quad\quad
 m\geq4,\ r\leq3.
\]
Moreover, there exists a rationally connected threefold $X$ such that
$G\subset\operatorname{Bir}(X)$ if and only if one of the following
holds:
\[
 m=2,\ r\leq6,\quad\quad
 m\in\{3,4\},\ r\leq4,\quad\quad
 m\geq5,\ r\leq3.
\]
\end{theorem}

Thus, among groups of the form $(\Z/m)^r$, the group $(\Z/4)^4$ is
the unique one that embeds into $\operatorname{Bir}(X)$ for some
rationally connected threefold $X$ but does not embed into
$\Cr_3(\C)$.

\textbf{Organization of the paper.}
Section~\ref{sec-preliminaries} collects the preliminary results on
finite group actions, equivariant Mori theory, K3 surfaces and their
invariant lattices, and projective representations.  Proposition
\ref{prop-mmp} reduces Theorem~\ref{thm-classification} to a terminal
$G\Q$-Fano threefold~$X$.  In Section~\ref{sec-gorenstein}, an
invariant anticanonical K3 surface gives
$(-K_X)^3\in\{4,16,36,64\}$.  Prokhorov's results, the classification
of anticanonical models of Gorenstein Fano threefolds and Namikawa's bound for the number of singular points show that $X$ is a smooth
quartic. Proposition~\ref{prop-fermat} identifies its equation and the
action of $G$.  Section~\ref{sec-nongorenstein} excludes the
non-Gorenstein case by combining divisibility of orbit lengths with an
equivariant Euler-characteristic congruence on a resolution and using the
Hirzebruch--Riemann--Roch formula.
In Section~\ref{sec-sharp-bounds} we show that the group $(\mathbb{Z}/6)^4$ cannot act faithfully on a rationally connected threefold, which is needed to prove Theorem~\ref{thm-sharp-bounds}. Finally, in Section \ref{sec-proof-main-results} we prove the main results.

\textbf{Acknowledgements.}
Version 5.6 of OpenAI's ChatGPT was used to explore proof
ideas, test preliminary arguments, and identify possible references.
The author independently checked every source and argument, wrote the
final text, and assumes full responsibility for its content.

\section{Preliminaries}
\label{sec-preliminaries}
We work over the field of complex numbers $\mathbb{C}$. 
We use the language of the minimal model program as in
\cite{KollarMori}.

\subsection{Group actions on varieties}
Let $G$ be a finite group. 
By a \emph{$G$-variety} we mean a variety
$X$ endowed with a (not necessarily faithful) action of $G$.
A $G$-variety $X$ is called $G\mathbb{Q}$-factorial if every
$G$-invariant Weil divisor on $X$ is $\mathbb{Q}$-Cartier. It is called
$G$-factorial if every $G$-invariant Weil divisor is Cartier.  
\begin{definition}
A
$G\mathbb{Q}$-\emph{Mori fiber space} is a $G$-variety together with a $G$-equivariant contraction
$f\colon X\to Z$ with $\dim Z<\dim X$, where $X$ is
$G\mathbb{Q}$-factorial and has terminal singularities, satisfies $\rho^G(X/Z)=1$ and $-K_X$ is ample over
$Z$.  If $Z$ is a point, we say that $X$ is a
$G\mathbb{Q}$-Fano variety.    If, in
addition, $X$ is $G$-factorial, we call it a $G$-Fano variety.
\end{definition}

For threefolds, a $G\Q$-Fano variety is $G$-Fano precisely when it is
Gorenstein.  Indeed, if $X$ is $G$-Fano, then the $G$-invariant divisor
$K_X$ is Cartier, so $X$ is Gorenstein.  Conversely, if $X$ is
Gorenstein, the torsion-freeness of the local divisor class group of
terminal Gorenstein threefold singularities implies that every
$G$-invariant $\Q$-Cartier Weil divisor is Cartier
\cite[Lemma~5.1]{Kawamata88}.

\begin{example}
\label{ex-standard-actions}
We collect some elementary examples of actions of finite abelian groups on rationally connected varieties. 
\begin{enumerate}
\item
For every $m$, one has
$(\Z/m)^3\subset\operatorname{Aut}(\PP^3)$.
\item
Taking the product of three copies of the Klein four-group action on
$\mathbb{P}^1$ gives a faithful action of $(\Z/2)^6$ on $(\PP^1)^3$.
\item
Let
\[
 S_3=\{x_0^3+x_1^3+x_2^3+x_3^3=0\}\subset\PP^3
\]
be the Fermat cubic surface.  The product of the diagonal action of
$(\Z/3)^3$ on $S_3$ with an action of a cyclic group of order $3$ on $\PP^1$ gives a faithful
action of $(\Z/3)^4$ on the rational threefold $S_3\times\PP^1$.
\item
The standard diagonal action of $(\Z/4)^4$ on the Fermat quartic
threefold $X_4$ as in \eqref{eq-fermat} is faithful.
\end{enumerate}
\end{example}

\subsection{Point stabilizers and orbit lengths}
\begin{lemma}[cf. {\cite[Lemma 4]{Po14}}]
\label{lem-faithful-action}
Let a finite group $G$ act faithfully on an algebraic variety $X$.
If $P\in X$ is a fixed point of $G$, then the induced action of $G$ on the
tangent space $T_PX$ is faithful.
\end{lemma}
\begin{theorem}[{\cite[Theorem 7.3]{LoginovProduct}}]
\label{thm-local}
Let $P\in X$ be a terminal threefold germ and let
$A\subset\Aut(P\in X)$ be a finite abelian subgroup.  Then either $A$
can be generated by at most three elements, or
\[
 A\simeq\Z/2n\times\Z/2m\times (\Z/2)^2
\]
for some positive integers $m,n$.  In the latter case $P\in X$ is a
Gorenstein singularity of type $cA$.
\end{theorem}

\begin{corollary}
\label{lem-orbits}
Assume that $G=(\Z/4)^4$ acts faithfully on a terminal threefold $X$. Then every $G$-orbit of closed points on $X$ has cardinality
divisible by four.
\end{corollary}

\begin{proof}
If a $G$-orbit had cardinality one or two, then its stabilizer would be
isomorphic to $(\Z/4)^4$ or $(\Z/4)^3\times\Z/2$, respectively.
Neither group occurs in Theorem~\ref{thm-local}. Hence every orbit has cardinality at least four.
Since the cardinality of every $G$-orbit is a power of two, it is divisible by four.
\end{proof}

The following lemma also follows from
\cite[Proposition~4.3]{Haution}. We give a proof for the reader's convenience.

\begin{lemma}
\label{lem-invariant-zero-cycle}
Let a finite abelian group $A$ act on a projective integral variety $V$
of dimension $n$.  If $L_1,\ldots,L_n$ are $A$-linearized line bundles,
then the class
\[
 c_1(L_1)\cdots c_1(L_n)\cap[V]\in\mathrm{CH}_0(V)
\]
is represented by an $A$-invariant zero-cycle.
\end{lemma}

\begin{proof}
We argue by induction on $n$.  The assertion is clear for $n=0$.
Choose a nonzero $A$-semi-invariant rational section $s$ of $L_1$.
Such a section exists because the finite-dimensional span of the orbit of any
rational section contains an eigenvector.  Put $D=\operatorname{div}(s)$.
The divisor $D$ is $A$-invariant and represents
$c_1(L_1)\cap[V]$.  Choose one prime divisor $W$ from each $A$-orbit
of the components of $D$ and let $A_W$ be the setwise stabilizer of
$W$.  We may then write
\[
 D=
 \sum_W m_W\sum_{aA_W\in A/A_W}aW.
\]
The restricted line bundles $L_2|_W,\ldots,L_n|_W$ are
$A_W$-linearized.  By the induction hypothesis, their intersection on
$W$ is represented by an $A_W$-invariant zero-cycle
$\zeta_W=\sum_x \nu_{W,x}[x]$.  For $h\in A_W$, one has
$h_*\zeta_W=\zeta_W$, and hence $a_*\zeta_W$ depends only on the coset
$aA_W$.  Therefore
\begin{equation}
\label{eq-0-cycle}
 c_1(L_1)\cdots c_1(L_n)\cap[V]
 =\sum_W m_W\sum_{aA_W\in A/A_W}a_*\zeta_W
\in \mathrm{CH}_0(V).
\end{equation}
The action of $A$ permutes the cosets in $A/A_W$, so the zero-cycle \eqref{eq-0-cycle} is $A$-invariant.
\end{proof}

\begin{corollary}
\label{cor-intersection-divisibility}
Let $G=(\Z/4)^4$ act faithfully on a projective terminal threefold
$V$.  If $L_1,L_2,L_3$ are $G$-linearized line bundles, then
$
 4\mid(L_1\cdot L_2\cdot L_3).
$
\end{corollary}

\begin{proof}
By Corollary~\ref{lem-orbits}, every $G$-orbit of closed points on
$V$ has cardinality divisible by four.  By
Lemma~\ref{lem-invariant-zero-cycle}, the intersection class is
represented by a $G$-invariant zero-cycle.  Its degree is therefore
divisible by four.
\end{proof}

\subsection{Reduction to Fano threefolds}

\begin{proposition}
\label{prop-mmp}
Let $X$ be a rationally connected threefold and let
$G=(\Z/4)^4\subset\operatorname{Bir}(X)$.  Then $X$ is
$G$-birational to a $G\Q$-Fano threefold.
\end{proposition}

\begin{proof}
By the equivariant minimal model program, the variety $X$ is $G$-birational
to a Mori fiber space $X'\to Z$.  If $\dim Z>0$, then $G$ is of product
type by \cite[Corollary 3.17]{LoginovProduct}, which contradicts the
fact that $G$ is not of product type
\cite[Table 1]{LoginovProduct}.  Hence
$\dim Z=0$, so $X'$ is a $G\Q$-Fano threefold.
\end{proof}

\subsection{K3 surfaces and invariant lattices}
\label{subsec-k3-surfaces}
Let $X$ be a $G\mathbb{Q}$-Fano threefold where $G=(\mathbb{Z}/4)^4$. 
Note that $G$ acts naturally on $H^0(X, -K_X)$. 
If $|-K_X|\neq\emptyset$, then the linear action of the abelian group
$G$ on $H^0(X,-K_X)$ has an eigenvector, and hence there is a
$G$-invariant element $S\in|-K_X|$.

\begin{lemma}
\label{lem-invariant-k3}
Let $X$ be a $G\mathbb{Q}$-Fano threefold where $G=(\mathbb{Z}/4)^4$.  Assume that $|-K_X|\neq \emptyset$. 
Let $S\in|-K_X|$ be a $G$-invariant element.  Then $(X,S)$ is purely log
terminal, and the minimal resolution $\widetilde S$ of $S$ is a smooth
K3 surface.
\end{lemma}

\begin{proof}
By \cite[Proof of Theorem~1.7]{LoginovProduct}, if $(X,S)$ were not
purely log terminal, then $G$ would be of product type, which is not the case.  Thus $(X,S)$ is
purely log terminal.  The surface $S$ has only Du Val singularities,
adjunction gives $K_S\sim0$, and Kawamata--Viehweg vanishing yields
$H^1(S,\mathcal O_S)=0$.  Hence $\widetilde S$ is a smooth K3 surface.
\end{proof}

Restriction of the $G$-action to $S$ gives an exact sequence
\begin{equation}
\label{eq-k3-exact-seq}
 1\longrightarrow C\longrightarrow G\longrightarrow H
 \longrightarrow1,
\end{equation}
where $C=\ker(G\to\Aut(S))$ and $H=G/C$ acts faithfully on $S$.
The group $C$ acts faithfully on the normal line to $S$ at a general
point and is therefore cyclic.  The $H$-action lifts naturally to the
minimal resolution $\widetilde S$.

\begin{proposition}[{\cite{BH}, \cite[Proposition~4.4]{LPZ}}]
\label{prop-bh-k3}
Let $H$ act faithfully on a smooth K3 surface $\widetilde{S}$.  Put
$M=H^2(\widetilde{S},\Z)^H$.  For the groups below, the invariant lattice is as
follows:
\begin{enumerate}
\item $H=(\Z/6)^2\times\Z/2$ and $M\simeq\langle2\rangle$,
\item $H=(\Z/4)^3$ and $M\simeq\langle4\rangle$,
\item $H=\Z/8\times\Z/4\times\Z/2$ and
$M\simeq\begin{pmatrix}0&2\\2&0\end{pmatrix}$,
\item $H=\Z/6\times(\Z/3)^2$ and
$M\simeq\begin{pmatrix}0&3\\3&0\end{pmatrix}$.
\end{enumerate}
\end{proposition}

\begin{proposition}
\label{prop-k3}
Let $X$ be a $G\mathbb{Q}$-Fano threefold where $G=(\mathbb{Z}/4)^4$.  Assume that $|-K_X|\neq \emptyset$. 
Let $S\in|-K_X|$ be a $G$-invariant element.  Let
$C$ and $H$ be as in \eqref{eq-k3-exact-seq}.  Then
\[
 C\simeq\Z/4,
 \qquad H\simeq(\Z/4)^3,
\]
the surface $S$ is smooth and does not meet $\Sing(X)$.  Moreover, $X$ is
Gorenstein.
\end{proposition}

\begin{proof}
By Lemma \ref{lem-invariant-k3}, the surface $S$ is a K3 surface with at worst du Val singularities. 
Since $C$ is cyclic, its order is at most four.  If $|C|\leq2$,
then $H$ contains a subgroup isomorphic to
$(\Z/4)^3\times\Z/2$.  If this group acted faithfully on a K3 surface,
it would be contained in a maximal finite abelian group acting
faithfully on that surface.  Since it is noncyclic, the latter group
would not be purely non-symplectic.  By \cite[Theorem~3.7]{LPZ}, every
such maximal group has order at most $72$, whereas
$|(\Z/4)^3\times\Z/2|=128$.  Hence
$C\simeq\Z/4$ and $H\simeq(\Z/4)^3$.
By Proposition~\ref{prop-bh-k3}(2), one has
$H^2(\widetilde S,\Z)^H\simeq\langle4\rangle$.
If $S$ had a Du Val
singularity, the class of the reduced union of an $H$-orbit of
exceptional curves on $\widetilde S$ would be nonzero and
$H$-invariant.  It has negative square because the lattice of exceptional curves
of the minimal resolution is negative definite.  This contradicts $H^2(\widetilde S,\Z)^H\simeq\langle4\rangle$. Thus $S$ is smooth. 

 Locally, a terminal Gorenstein threefold
singularity is a hypersurface germ.  At a singular point its defining
equation has no linear term, so any
Cartier divisor through that point is singular there.  On the other
hand, every non-Gorenstein point lies on every
anticanonical member and is a singular point thereof, see e.g.
\cite[Lemma~5.5]{LPZ}.  Hence $S$ does not meet the singular locus of
$X$.  In particular, $X$ has no non-Gorenstein points, and therefore
$X$ is Gorenstein.
\end{proof}

\subsection{Two representation lemmas}

\begin{lemma}
\label{lem-pgl4}
Let $G=(\mathbb{Z}/4)^4$.
Let $C\subset G$ be a subgroup of order at most two.  Then $G/C$ does not embed
into $\operatorname{PGL}_4(\C)$.
\end{lemma}

\begin{proof}
Put $A=G/C$.  Then $|A|\geq128$.  We have $\exp(A)=4$.  Suppose that
$A\subset\operatorname{PGL}(V)$, where
$\dim V=4$.  Since $\operatorname{PGL}(V)=\operatorname{PSL}(V)$, the
inverse image $\widetilde A$ of $A$ in $\operatorname{SL}(V)$ fits into
an exact sequence
\[
 1\longrightarrow\mu_4\longrightarrow\widetilde A
 \longrightarrow A\longrightarrow1.
\]
For $a,b\in A$, choose lifts $\widetilde a,\widetilde b\in\widetilde A$
and set
$
 e(a,b)=\widetilde a\widetilde b\widetilde a^{-1}
 \widetilde b^{-1}\in\mu_4.
$
Since this extension is central, $e(a,b)$ does not depend on the
chosen lifts.  Put
\[
 R=\{a\in A\mid e(a,b)=1\text{ for every }b\in A\}.
\]
By \cite[Chapter~8, Theorem~2.21(ii)]{Karpilovsky} we see that the
$\widetilde A$-module $V$ is a direct sum of irreducible submodules of
the same dimension $d$, where
$
 d^2=|A/R|.
$
For convenience, we give a proof of this fact.  Let $W\subset V$ be
an irreducible $\widetilde A$-submodule, put $d=\dim W$, and let $\chi$
be its character.  If $a\notin R$, choose $b\in A$ such that
$e(a,b)=\zeta\ne1$.  Then
\[
 \widetilde b\widetilde a\widetilde b^{-1}
 =\zeta^{-1}\widetilde a,
\]
and hence $\chi(\widetilde a)=\zeta^{-1}\chi(\widetilde a)=0$.
If $a\in R$, every lift of $a$ is central in $\widetilde A$ and
therefore acts on $W$ by a scalar of absolute value one.  Thus
$|\chi(\widetilde a)|=d$.  Character orthogonality gives
\[
 1=\frac{1}{|\widetilde A|}
 \sum_{g\in\widetilde A}|\chi(g)|^2
 =\frac{|\mu_4||R|d^2}{|\mu_4||A|},
\]
so $d^2=|A/R|$.  Since $W$ was arbitrary, complete reducibility shows
that $V$ is a direct sum of irreducible submodules of dimension $d$.

Thus $d\in\{1,2,4\}$, and the number of summands is $k=4/d$.  The
lift of every element $r\in R$ is central in $\widetilde A$, and hence
acts on the $i$-th summand by a scalar $\lambda_i(r)$.  Since $r^4=1$
in $A$, the fourth power of any lift of $r$ belongs to the central
subgroup $\mu_4$.  Consequently the scalars $\lambda_i(r)^4$ are
independent of $i$, and
\[
 \frac{\lambda_i(r)}{\lambda_1(r)}\in\mu_4.
\]
For $i=2,\ldots,k$, the ratio $\lambda_i/\lambda_1$ is a character
of $R$ with values in $\mu_4$ which is independent of the chosen lifts.
These $k-1$ characters separate the elements of $R$, since an element
on which they all take the value $1$ acts by a scalar on $V$ and is
therefore trivial in $A\subset\operatorname{PGL}(V)$.  Thus $R$ embeds
into $\mu_4^{k-1}$, and hence $|R|\leq4^{k-1}$.
Consequently
\[
 |A|=|A/R||R|\leq d^2 4^{4/d-1}\leq64,
\]
contrary to $|A|\geq128$.
\end{proof}
\begin{lemma}
\label{lem-quadric}
Let $C\subset G$ have order at most two.  Then $G/C$ does not act
faithfully on a smooth quadric threefold $Q\subset \mathbb{P}^4$.
\end{lemma}

\begin{proof}
One has
\[
 \Aut(Q)=\operatorname{PSO}_5(\C)
 \simeq\operatorname{PSp}_4(\C)
 \subset\operatorname{PGL}_4(\C).
\]
The assertion therefore follows from Lemma \ref{lem-pgl4}.
\end{proof}

\section{The Gorenstein case}
\label{sec-gorenstein}
Let $X$ be a $G\mathbb{Q}$-Fano threefold where $G=(\mathbb{Z}/4)^4$. 
Assume that $X$ is Gorenstein.  By Riemann--Roch and
Kawamata--Viehweg vanishing, we have
\[
h^0(X,-K_X)=\frac12(-K_X)^3+3>0.
\]
Since $G$ acts on $H^0(X,-K_X)$, there is a $G$-invariant element
$S\in|-K_X|$.  By Proposition~\ref{prop-k3}, the surface $S$ is a smooth
K3 surface and $S\subset X_{\mathrm{reg}}$.

\begin{proposition}
\label{prop-degree}
Let $X$ be a $G$-Fano threefold (so that $X$ is Gorenstein) where $G=(\mathbb{Z}/4)^4$. 
One has
\[
 (-K_X)^3\in\{4,16,36,64\}.
\]
\end{proposition}

\begin{proof}
Let $S\in |-K_X|$ be a $G$-invariant smooth K3 surface.  
Put $C=\ker(G\to\Aut(S))$ and $H=G/C$.  
Proposition~\ref{prop-k3} gives $H\simeq(\Z/4)^3$, while
Proposition~\ref{prop-bh-k3}(2) shows that the invariant lattice
$H^2(S,\Z)^H$ is generated by an ample class $D$ with $D^2=4$.
Thus $-K_X|_S=kD$ for some $k>0$,
and
\[
 (-K_X)^3=(-K_X|_S)^2=4k^2.
\]
For Gorenstein Fano threefolds one has
$(-K_X)^3\leq64$. Hence
$1\leq k\leq4$. The claim follows.
\end{proof}

\subsection{Reduction to Picard number one}

We use the following numerical consequence of
\cite[Theorems 1.2 and 6.5]{ProkhorovGFanoII}.

\begin{theorem}
\label{thm-prokhorov}
Let $X$ be a terminal Gorenstein Fano threefold endowed with an action
of a finite group $\Gamma$.  Assume that
\[
 \rk\Cl(X)^\Gamma=1,
 \qquad \rho(X)>1.
\]
Then
\[
 \bigl(\rho(X),(-K_X)^3\bigr)\in
 \left\{
 \begin{array}{c}
 (2,12),(2,20),(2,28),(2,48),\\
 (3,12),(3,30),(3,48),(4,24)
 \end{array}
 \right\}.
\]
\end{theorem}

\begin{corollary}
\label{cor-rho-one}
A $G$-Fano threefold $X$ where $G=(\mathbb{Z}/4)^4$ has Picard number one.
\end{corollary}

\begin{proof}
Since $X$ is $G$-Fano, one has $\rho(X)^G=1$.  None of the degrees
in Theorem \ref{thm-prokhorov} belongs to
$\{4,16,36,64\}$, contradicting Proposition \ref{prop-degree}.
\end{proof}

\subsection{The anticanonical map}

\begin{proposition}
\label{prop-anticanonical-free}
Let $X$ be a $G$-Fano threefold where $G=(\mathbb{Z}/4)^4$. 
Then the linear system $|-K_X|$ is basepoint-free.
\end{proposition}

\begin{proof}
By
\cite[Proposition~2.4.1, Remark 2.4.2]{IP}, if 
$\Bs|-K_X|\neq\varnothing$ then $\Bs|-K_X|$ is either a single
point or isomorphic to $\PP^1$.  In the first case $G$ has a fixed
point, in the second case $X$ contains a $G$-invariant rational curve.
These two cases are excluded by Corollary \ref{lem-orbits} and
\cite[Corollary~3.14]{LoginovProduct}, respectively.
\end{proof}

For a Gorenstein Fano threefold with a basepoint-free anticanonical
linear system, the anticanonical morphism is the map
\[
 \varphi=\varphi_{|-K_X|}\colon X\longrightarrow\PP^{g(X)+1},
 \qquad (-K_X)^3=2g(X)-2.
\]
The number $g(X)$ is called the \emph{genus} of $X$.  The \emph{Fano
index} $i(X)$ is the largest positive integer such that $-K_X$ is
divisible by $i(X)$ in $\Pic(X)$.  If $i(X)=1$, then
$2\leq g(X)\leq12$ and $g(X)\neq11$.

\begin{proposition}
\label{prop-gorenstein-list}
Let $X$ be a $G$-Fano threefold where $G=(\mathbb{Z}/4)^4$. 
One of the following holds:
\begin{enumerate}
\item $X\simeq\PP^3$,
\item $X$ is a quartic double solid,
\item $X$ is a double cover of a quadric threefold,
\item $X\subset\PP^4$ is a quartic hypersurface,
\item $X\subset\PP^{10}$ is a prime Fano threefold of genus $9$.
\end{enumerate}
\end{proposition}

\begin{proof}
If $i(X)>2$, apply
\cite[Section~3.1]{ProkhorovGFanoI}. If $i(X)=2$, then
\cite[Theorem~3.2]{ProkhorovGFanoI} applies. Together with
Proposition \ref{prop-degree}, these results leave only
$X\simeq\PP^3$ and the del Pezzo threefold of degree two.  The latter
is a double cover of $\PP^3$ branched over a quartic surface and is
called a quartic double solid.

Assume that $i(X)=1$.  The bound $g(X)\leq12$ excludes degrees $36$
and $64$.  If $(-K_X)^3=4$, then $g(X)=3$.  If $-K_X$ is very ample,
its image in $\PP^4$ is a quartic hypersurface.  Otherwise,
\cite[Theorem~2.12]{PCS} shows that the anticanonical morphism is a
double cover of a variety of minimal degree, in this case the latter
is a quadric threefold in $\PP^4$.  If $(-K_X)^3=16$, then $g(X)=9$,
and \cite[Theorem~4.2]{ProkhorovRationalityI} shows that $-K_X$ is very
ample.  Thus the anticanonical map embeds $X$ as a prime Fano
threefold of genus $9$ in $\PP^{10}$.
\end{proof}

\begin{proposition}
\label{prop-hyperelliptic}
Let $X$ be a $G$-Fano threefold where $G=(\mathbb{Z}/4)^4$. 
Then $X$ is neither $\PP^3$ nor a quartic double solid nor a
double cover of a quadric threefold.
\end{proposition}

\begin{proof}
The case $X=\PP^3$ is excluded by Lemma \ref{lem-pgl4}.

Let $X$ be a quartic double solid.  Its fundamental linear system
$|L|$, where $-K_X=2L$, defines the double cover $X\to\PP^3$.  Let
$\iota$ be its deck involution and put
$C=\ker(G\to\Aut(\PP^3))$.  Then $C\subset\langle\iota\rangle$, so
$|C|\leq2$, and $G/C$ acts faithfully on $\PP^3$.  This contradicts
Lemma~\ref{lem-pgl4}.

Let $X$ be a double cover of a quadric threefold $Q$, let $\iota$ be
the deck involution, and put $C=\ker(G\to\Aut(Q))$.  Then
$C\subset\langle\iota\rangle$, so $|C|\leq2$, and $G/C$ acts
faithfully on $Q$. If $Q$ is smooth, Lemma
\ref{lem-quadric} applies.  If the quadric is singular, terminality of
$X$ forces it to have rank four.  Indeed, the inverse image of
$\Sing(Q)$ is contained in $\Sing(X)$, and terminal threefold
singularities are isolated, so $\Sing(Q)$ is zero-dimensional.  Thus $Q$
is a cone with vertex $P$.  The point $P$ is fixed by $G/C$, and its
inverse image consists of at most two points of $X$.  This gives a
$G$-orbit of cardinality at most two, contradicting Corollary
\ref{lem-orbits}.
\end{proof}

\subsection{Singular anticanonical models}
Let $X$ be a $G$-Fano threefold where $G=(\mathbb{Z}/4)^4$.  If $X$ is
singular, then by \cite[Theorem~11]{Namikawa} it admits a smoothing
$X_t$, and \cite[Theorem~13]{Namikawa} gives
\begin{equation}
 |\Sing(X)|\leq20-\rho(X_t)+h^{1,2}(X_t).
 \label{eq-namikawa}
\end{equation}

\begin{proposition}
\label{prop-embedded-smooth}
The quartic and the genus-nine threefold in Proposition
\ref{prop-gorenstein-list} are smooth.
\end{proposition}

\begin{proof}
Suppose that $P\in\Sing(X)$.  Choose an eigenbasis
$s_0,\ldots,s_N$ of $H^0(X,-K_X)$.  Each divisor $(s_i=0)$ is
$G$-invariant and hence, by Proposition~\ref{prop-k3}, does not
contain $P$.  Thus $s_i(P)\neq0$ for every $i$.

We can identify
$\Pic(X)$ with $\Pic(X_t)$, in particular $\rho(X_t)=\rho(X)=1$, cf.
\cite[Theorem~1.4]{JahnkeRadloff}.
The canonical class and its cube are preserved under the smoothing.
Moreover, after shrinking the base, relative very ampleness of the
anticanonical bundle shows that $X_t$ is a smooth quartic threefold in
the quartic case and a smooth prime Fano threefold of genus $9$ in the
other case.  The corresponding values of $h^{1,2}(X_t)$ are $30$ and
$3$, respectively, see \cite[Table~12.2]{IP}.  Hence
\eqref{eq-namikawa} gives $|\Sing(X)|\leq49$ in the first case and
$|\Sing(X)|\leq22$ in the second.  Since
$G\cdot P\subset\Sing(X)$, in either case
$|G\cdot P|<256=|G|$.  Thus the stabilizer $G_P$ contains a nontrivial
element $g$.

In $\mathbb{P}(H^0(X,-K_X)^\vee)$, we have
\[
 P=[s_0(P):\ldots:s_N(P)]
\]
with all coordinates nonzero.  Since $g$ fixes $P$ and acts diagonally
in the chosen coordinates, its diagonal entries are equal.  Hence $g$
acts trivially on the ambient projective space, contrary to
faithfulness.  Thus $X$ is smooth.
\end{proof}

\subsection{The projective model}

\begin{proposition}
\label{prop-smooth-quartic}
Let $X$ be a $G$-Fano threefold where $G=(\mathbb{Z}/4)^4$. 
Then $X$ is a smooth quartic hypersurface in $\PP^4$.
\end{proposition}

\begin{proof}
By Propositions \ref{prop-hyperelliptic} and
\ref{prop-embedded-smooth}, $X$ is either a smooth quartic threefold or
a smooth prime Fano threefold of genus $9$.  
Let $S$ be a $G$-invariant smooth K3 surface. 
Write $\Pic(X)=\Z L$ and
$-K_X=i(X)L$, where $L$ is primitive.  Since $S$ is a smooth ample
divisor in the smooth threefold $X$, the integral Lefschetz theorem
shows that the restriction map $\Pic(X)\to\Pic(S)$ is injective with
torsion-free cokernel, see
\cite[Corollary~2.3.4]{BeltramettiSommese}.  Hence $L|_S$ is
primitive.
By Proposition~\ref{prop-k3}, the group acting faithfully on $S$ is
$H\simeq(\Z/4)^3$.  Since $L|_S$ is a primitive $H$-invariant class,
Proposition~\ref{prop-bh-k3}(2) gives $(L|_S)^2=4$.  Hence
\[
 i(X)L^3=(L|_S)^2=4.
\]
For a Fano threefold of genus $9$ one has $i(X)=1$ and $L^3=16$, a
contradiction.  Thus $X\subset\PP^4$ is a smooth quartic hypersurface.
\end{proof}

\begin{proposition}
\label{prop-fermat}
Let $X\subset\PP^4$ be a smooth quartic threefold endowed with a
faithful action of $G=(\Z/4)^4$.  Then, up to an automorphism of $G$,
there is a $G$-equivariant projective isomorphism $X\simeq X_4$, where
$G$ acts on $X_4$ by the standard diagonal action \eqref{eq-fermat}.
\end{proposition}

\begin{proof}
The natural $G$-linearization of $-K_X$ gives a representation on
$V=H^0(X,-K_X)$ whose projectivization is faithful.  Write $G^\vee$
additively and choose an eigenbasis
$x_0,\ldots,x_4$ with characters $\chi_0,\ldots,\chi_4$, and put
$\varepsilon_i=\chi_i-\chi_0$ for $1\leq i\leq4$.  These four
characters generate $G^\vee$: an element annihilated by all of them
would act on $V$ by a scalar.  Hence they form a basis of
$G^\vee\simeq(\Z/4)^4$, which we may take to be the standard basis
after an automorphism of $G$.

The equation  $F$ of $X$ is semi-invariant.  The weight of a monomial
$x_0^{e_0}\ldots x_4^{e_4}$ is
$\sum_{i=0}^4e_i\chi_i=\sum_{i=1}^4e_i\varepsilon_i$, because
$\sum e_i=4$ and $4\chi_0=0$.  A nonzero weight occurs in at most one
quartic monomial: for such a weight all $e_i$ with $i\geq1$ are less
than four, so their residues modulo four determine them, and then
$e_0$ is determined by $\sum e_i=4$.  Hence a
semi-invariant quartic of nonzero weight would be a monomial, which
cannot define a smooth hypersurface.  Therefore $F$ has weight zero.
The weight-zero quartic monomials are precisely
$x_0^4,\ldots,x_4^4$, and hence
$F=a_0x_0^4+\ldots+a_4x_4^4$.  Smoothness implies that every $a_i$ is
nonzero.  Rescaling the coordinates gives the Fermat equation.
\end{proof}

\section{The non-Gorenstein case}
\label{sec-nongorenstein}
Let $X$ be a $G\mathbb{Q}$-Fano threefold where $G=(\mathbb{Z}/4)^4$. 
Assume that $X$ is non-Gorenstein.  By Proposition~\ref{prop-k3}, we have
$
 H^0(X,-K_X)=0.
$

\subsection{Basket of singularities}
\label{sec-about-baskets}
Let $P\in X$ be a terminal threefold singularity.  A small deformation
of the germ $P\in X$ has finitely many terminal cyclic quotient
singularities
\[
 Q_i=\frac1{r_i}(a_i,-a_i,1).
\]
Their multiset is called the \emph{local Reid basket} of $P\in X$.
After grouping equal quotient types, we write
\begin{equation}
 B(P\in X)=
 \left\{n_i\times\frac1{r_i}(a_i,-a_i,1)\right\}_{i=1}^s,
 \label{eq-local-basket}
\end{equation}
where the displayed quotient types are pairwise distinct and $n_i$ is
their multiplicity.
The global Reid basket $B(X)$ is the disjoint union of the local
baskets over all non-Gorenstein points of $X$, see
\cite[Section~6]{Reid}.

If $P\in X$ has index $r$, then every entry of $B(P\in X)$ has index
$r$, except when $P\in X$ is of type $cAx/4$.  In that case the basket
contains one index-four entry and all its remaining entries have index
two.  If $P\in X$ is not itself a cyclic quotient singularity, then
its local basket contains at least two entries.

\subsection{Orbifold Riemann--Roch formula}
Consider a basket element
$
 Q=\frac1r(a,-a,1).
$
Interchanging $a$ and $-a$ does not change $Q$.  We therefore choose
the sign of $a$ so that its inverse $b$ modulo $r$ satisfies
$1\leq b\leq r/2$. Thus $ab\equiv1\pmod r$.  Set
\[
 d(Q)=\frac{b(r-b)}r,
 \qquad c(Q)=r-\frac1r.
\]
Reid's orbifold Riemann--Roch formula applies to a $\Q$-Cartier Weil
divisor on a terminal threefold \cite[Section~10.2]{Reid}, see also
\cite[Theorem~12.1.3]{ProkhorovEMMP}.  
Thus, for a $G\Q$-Fano threefold one has
\[
 (-K_X)\cdot c_2(X)+
 \sum_{Q\in B(X)}\left(r_Q-\frac1{r_Q}\right)=24
\]
and
\[
 (-K_X)\cdot c_2(X)>0.
\]
If $H^0(X,-K_X)=0$, Kawamata--Viehweg vanishing gives
$\chi\bigl(X,\mathcal O_X(-K_X)\bigr)=0$.  Substituting $D=-K_X$ in orbifold
Riemann--Roch and using the displayed identity gives
\begin{align}
 (-K_X)^3&=-6+\sum_{Q\in B(X)}d(Q),
 \label{eq-orr-volume}\\
 (-K_X)\cdot c_2(X)&=
 24-\sum_{Q\in B(X)}c(Q)>0.
 \label{eq-orr-c2}
\end{align}

\begin{remark}
For a $G\mathbb{Q}$-Fano threefold where $G=(\mathbb{Z}/4)^4$, the
basket can be determined from \eqref{eq-orr-volume} and
\eqref{eq-orr-c2}.  The result is
\[
 B(X)=8\times\frac12(1,1,1)
 \ \sqcup\
 4\times\frac13(1,2,1),
 \qquad
 (-K_X)^3=\frac23.
\]
Indeed, Corollary~\ref{lem-orbits} shows that every basket
multiplicity is divisible by four.  Writing it as $4x_{r,b}$ gives
\[
 \sum x_{r,b}\frac{b(r-b)}r>\frac32,
 \qquad
 \sum x_{r,b}\left(r-\frac1r\right)<6.
\]
The second inequality gives $r\leq6$.  Checking the finite multisets of
coprime pairs $(r,b)$ satisfying both inequalities gives
$x_{2,1}=2$ and $x_{3,1}=1$. Then \eqref{eq-orr-volume} gives the
degree.
\end{remark}

\subsection{Equivariant Euler-characteristic congruence}

\begin{lemma}
\label{lem-euler-congruence}
Let a finite group $A$ act on a normal projective variety $X$ with
isolated singularities, and let $N\geq1$ be an integer such that the
cardinality of every $A$-orbit in $X$ is divisible by $N$.  Let
$\pi\colon Z\to X$ be an $A$-equivariant resolution which is an
isomorphism over $X_{\mathrm{reg}}$.  Then
\[
 \chi\bigl(Z,\mathcal O_Z(-K_Z)\bigr)\equiv
 \chi\bigl(X,\mathcal O_X(-K_X)\bigr)\pmod N.
\]
\end{lemma}

\begin{proof}
The sheaf $\pi_*\omega_Z^{-1}$ is torsion-free of rank one, and
\[
 (\pi_*\omega_Z^{-1})^{**}\simeq\omega_X^{[-1]}
 =\mathcal O_X(-K_X),
\]
because the two sheaves agree over $X_{\mathrm{reg}}$.  Hence the
cokernel $\mathcal Q$ in the exact sequence
\begin{equation}
 0\longrightarrow\pi_*\omega_Z^{-1}
 \longrightarrow\mathcal O_X(-K_X)
 \longrightarrow\mathcal Q\longrightarrow0
 \label{eq-anticanonical-cokernel}
\end{equation}
has finite length.  The higher direct images
$R^i\pi_*\omega_Z^{-1}$ for $i>0$ are also supported on $\Sing(X)$ and
have finite length.  The Leray spectral sequence
\[
 E_2^{p,i}=H^p\bigl(X,R^i\pi_*\omega_Z^{-1}\bigr)
 \quad\Longrightarrow\quad
 H^{p+i}(Z,\omega_Z^{-1})
\]
gives, after taking Euler characteristics
(cf.~\cite[Tag~0EDD, Lemma~72.17.3]{StacksProject}),
\begin{equation}
 \chi\bigl(Z,\mathcal O_Z(-K_Z)\bigr)=
 \sum_{i\geq0}(-1)^i
 \chi\bigl(X,R^i\pi_*\omega_Z^{-1}\bigr).
 \label{eq-leray-euler-characteristic}
\end{equation}
By additivity of the Euler characteristic in
\eqref{eq-anticanonical-cokernel},
\begin{equation}
 \chi\bigl(X,\pi_*\omega_Z^{-1}\bigr)
 =\chi\bigl(X,\mathcal O_X(-K_X)\bigr)-\ell(\mathcal Q).
 \label{eq-euler-pushforward}
\end{equation}
Moreover, for $i>0$,
\begin{equation}
 \chi\bigl(X,R^i\pi_*\omega_Z^{-1}\bigr)
 =\ell\bigl(R^i\pi_*\omega_Z^{-1}\bigr),
\label{eq-euler-higher-direct-images}
\end{equation}
since these sheaves have zero-dimensional support.  Substituting
\eqref{eq-euler-pushforward} and
\eqref{eq-euler-higher-direct-images} into
\eqref{eq-leray-euler-characteristic} gives
\begin{equation}
 \chi\bigl(Z,\mathcal O_Z(-K_Z)\bigr)
 -\chi\bigl(X,\mathcal O_X(-K_X)\bigr)
 =-\ell(\mathcal Q)+
 \sum_{i>0}(-1)^i
 \ell\bigl(R^i\pi_*\omega_Z^{-1}\bigr).
 \label{eq-leray-lengths}
\end{equation}
All the sheaves on the right-hand side are naturally
$A$-equivariant.  Their stalk lengths are constant along each
$A$-orbit, and therefore their total lengths are divisible by $N$.
The congruence follows from \eqref{eq-leray-lengths}.
\end{proof}

\begin{corollary}
\label{cor-nongorenstein}
There is no non-Gorenstein terminal $G\Q$-Fano threefold where $G=(\mathbb{Z}/4)^4$.
\end{corollary}

\begin{proof}
Choose a smooth projective $G$-equivariant resolution
$\pi\colon Z\to X$ which is an isomorphism over $X_{\mathrm{reg}}$,
see \cite[Theorem~1.1]{BierstoneMilman}.
Both $X$ and $Z$ are terminal threefolds with faithful $G$-actions.
Hence, by Corollary \ref{lem-orbits}, every $G$-orbit in either variety
has cardinality divisible by four.
Since $H^0(X,-K_X)=0$, Kawamata--Viehweg vanishing gives
$
 \chi\bigl(X,\mathcal O_X(-K_X)\bigr)=0.
$
Hence Lemma \ref{lem-euler-congruence} yields
\begin{equation}
 \chi\bigl(Z,\mathcal O_Z(-K_Z)\bigr)\equiv0\pmod4.
 \label{eq-euler-mod-four}
\end{equation}

The anticanonical bundle of $Z$ is naturally $G$-linearized, so
Corollary \ref{cor-intersection-divisibility} gives
$
 4\mid(-K_Z)^3.
$
The smooth threefold $Z$ is rationally connected, and hence
$\chi(Z,\mathcal O_Z)=1$.  For a smooth threefold,
Hirzebruch--Riemann--Roch gives
\[
 \chi\bigl(Z,\mathcal O_Z(-K_Z)\bigr)
 =\chi(Z,\mathcal O_Z)+\frac12(-K_Z)^3
  +\frac1{12}(-K_Z)\cdot c_2(Z),
\]
whereas
\[
 \chi(Z,\mathcal O_Z)=\frac1{24}(-K_Z)\cdot c_2(Z).
\]
Consequently,
\[
 \chi\bigl(Z,\mathcal O_Z(-K_Z)\bigr)
 =3+\frac12(-K_Z)^3.
\]
The right-hand side is odd, contradicting
\eqref{eq-euler-mod-four}.
\end{proof}

\section{Primary components and the case $(\mathbb Z/6)^4$}
\label{sec-sharp-bounds}

The following proposition is an immediate consequence of Theorem
\ref{thm-kollar-zhuang}, applied to the primary components of a finite
abelian subgroup $A\subset\operatorname{Bir}(X)$.

\begin{proposition}
\label{prop-abelian-structure-general}
Let $X$ be a rationally connected variety of dimension $n$ and let
$A\subset\operatorname{Bir}(X)$ be a finite abelian group.  Then
\[
 A\simeq A_0\times B,
\]
where $A_0$ can be generated by at most $n$ elements and
\begin{equation}
 |B|\mid\mathfrak h_n,
 \qquad
 \mathfrak h_n:=
 \prod_{\substack{p\text{ prime}\\p\leq n+1}}
 p^{\left\lfloor n/(p-1)\right\rfloor}.
 \label{eq-general-bound}
\end{equation}
\end{proposition}

\begin{proof}
After passing to an equivariant smooth projective model, write
\[
 A=\prod_p A_p.
\]
Theorem \ref{thm-kollar-zhuang} gives decompositions
\[
 A_p\simeq A'_p\times B_p,
 \qquad
 \rk A'_p\leq n,
 \qquad
 |B_p|\leq p^{n/(p-1)}.
\]
Put $A_0=\prod_p A'_p$ and $B=\prod_p B_p$.  Since the primary
factors have pairwise coprime orders, the minimal number of generators
of $A_0$ is the maximum of those of the groups $A'_p$.  Thus $A_0$ can
be generated by at most $n$ elements.  Since $|B_p|$ is a power of
$p$, we have
\[
 |B_p|\mid p^{\left\lfloor n/(p-1)\right\rfloor}.
\]
In particular, $B_p=\{1\}$ for $p>n+1$, and
\eqref{eq-general-bound} follows.
\end{proof}

\begin{corollary}
\label{cor-kollar-zhuang-uniform}
Let $X$ be a rationally connected threefold and let
$A\subset\operatorname{Bir}(X)$ be a finite abelian group.  Then
$
 A\simeq A_0\times B,
$
where $A_0$ can be generated by at most three elements and
$|B|$ divides~$24$.  In particular, if
$(\Z/m)^r\subset\operatorname{Bir}(X)$, where $m\geq2$ and $r\geq4$,
then $m^{r-3}$ divides $24$.
\end{corollary}

\begin{proof}
The first assertion follows from Proposition
\ref{prop-abelian-structure-general}, since
$\mathfrak h_3=2^3\cdot3=24$.  Now let $A=(\Z/m)^r$.  The group $A_0$
has exponent dividing $m$ and can be generated by at most three
elements, so $|A_0|\mid m^3$.  Since
$m^r=|A|=|A_0||B|$, it follows that
$m^{r-3}$ divides $|B|$, which in turn divides~$24$.
\end{proof}

\begin{remark}
The integer $\mathfrak h_n$ is the $n$-th \emph{Hirzebruch number},
see \cite{BuchstaberVeselov}.
Equivalently, it is the universal denominator of the $n$-th Todd
polynomial $T_n$, that is, the least positive integer $N$ such that
$NT_n$ has integral coefficients.  Thus
\[
 \mathfrak h_1=2,\qquad
 \mathfrak h_2=12,\qquad
 \mathfrak h_3=24,\qquad
 \mathfrak h_4=720,\ldots.
\]
Its appearance here reflects the Todd-class denominator in the proof
of \cite[Proposition~8]{KZh24}.
\end{remark}

Suppose that $G=(\Z/m)^4\subset\operatorname{Bir}(X)$ for a
rationally connected threefold $X$.  Corollary
\ref{cor-kollar-zhuang-uniform} shows that $m$ divides $24$, and hence
\[
 m\in\{2,3,4,6,8,12,24\}.
\]
The cases $m\in\{2,3,4\}$ are realized by Example
\ref{ex-standard-actions}, while for $m\in\{8,12,24\}$ the group
$(\Z/m)^4$ contains $(\Z/4)^4$.  The latter cases are excluded by
Theorem \ref{thm-classification} together with
$(\Z/4)^4\rtimes\mathfrak S_5=\operatorname{Bir}(X_4)=\operatorname{Aut}(X_4)$, since this group
contains none of $(\Z/m)^4$ for $m\in\{8,12,24\}$.  Thus $m=6$ is
the only case requiring a separate argument.

\begin{lemma}
\label{lem-no-six-four}
There is no rationally connected threefold $X$ such that
$(\Z/6)^4\subset\operatorname{Bir}(X)$.
\end{lemma}

\begin{proof}
Put $A=(\Z/6)^4$.  The group $A$ is not of product type by
\cite[Table 1]{LoginovProduct}, and it is not isomorphic to any of the
groups in \eqref{eq-exceptional-groups}.  Hence Theorem \ref{thm-three-types}  and Theorem \ref{thm-intro-2} show that a
faithful action of $A$ would give a terminal $A\Q$-Fano threefold $X$
with $|-K_X|=\varnothing$.

The threefold $X$ is non-Gorenstein.  Indeed, for a terminal
Gorenstein Fano threefold, Riemann--Roch and Kawamata--Viehweg
vanishing give
$
 h^0(X,-K_X)>0.
$
Let $P\in X$ be a non-Gorenstein point and let $A_P$ be its
stabilizer.  The exceptional alternative in Theorem
\ref{thm-local} occurs only at a Gorenstein singularity.  Therefore
$A_P$ can be generated by at most three elements.  The $2$-primary
and $3$-primary components of $A_P$ have orders at most $2^3$ and
$3^3$, respectively.  Thus
\[
 6\mid[A:A_P].
\]
Local Reid baskets are constant on $A$-orbits.  Hence the multiplicity
of every basket type $(r,b)$ in the global Reid basket $B(X)$ is
divisible by six. Write this multiplicity as $6x_{r,b}$, where
$x_{r,b}\geq0$.  Equations \eqref{eq-orr-volume} and
\eqref{eq-orr-c2} give
\begin{equation}
 \sum x_{r,b}\frac{b(r-b)}r>1,
 \qquad
 \sum x_{r,b}\left(r-\frac1r\right)<4.
 \label{eq-six-basket}
\end{equation}
The second inequality leaves only the following possible
contributions:
\[
\begin{array}{c|c|c}
(r,b)&b(r-b)/r&r-1/r\\ \hline
(2,1)&1/2&3/2\\
(3,1)&2/3&8/3\\
(4,1)&3/4&15/4
\end{array}
\]
An index-four term cannot occur together with any other term, and its
contribution to the first sum in \eqref{eq-six-basket} is $3/4$.  The
same argument applies to an index-three term, whose contribution to
the first sum is $2/3$.  If all terms have index two, the second
inequality allows at most two of them, whose total contribution to the
first sum is at most $1$.  This contradicts the first inequality in
\eqref{eq-six-basket}.
\end{proof}

\section{Proof of the main results}
\label{sec-proof-main-results}
\begin{proof}[Proof of Theorem \ref{thm-classification}]
By Proposition \ref{prop-mmp}, the variety $X$ is $G$-birational to a
terminal $G\Q$-Fano threefold $Y$.  Corollary
\ref{cor-nongorenstein} excludes the non-Gorenstein case.  Thus $Y$
is Gorenstein, and hence $G$-Fano.  By Proposition
\ref{prop-smooth-quartic}, it is a smooth quartic hypersurface.
Proposition \ref{prop-fermat} identifies $(Y,G)$, up to an
automorphism of $G$ and a projective change of coordinates, with the
Fermat quartic endowed with the standard diagonal $G$-action.  If $X$
is itself a terminal $G\Q$-Fano threefold with a regular action, the
same propositions apply directly to $X$ and give a $G$-equivariant
isomorphism $X\simeq X_4$.
\end{proof}

\begin{proof}[Proof of Corollary \ref{cor-cremona}]
An embedding $G\hookrightarrow\Cr_3(\C)=\operatorname{Bir}(\PP^3)$
would imply, by Theorem \ref{thm-classification}, that $\PP^3$ is birational
to $X_4$, contrary to the non-rationality of a smooth quartic
threefold \cite{IskovskikhManin}.
\end{proof}

\begin{proof}[Proof of Corollary \ref{cor-unique-gamma-model}]
Restricting the action to $(\Z/4)^4$ and applying Theorem
\ref{thm-classification}, we obtain a birational map
$\varphi\colon X\dashrightarrow X_4$.  Conjugation by $\varphi$ gives
an embedding
\[
 \Gamma\hookrightarrow\operatorname{Bir}(X_4),
 \qquad
 g\longmapsto\varphi\circ g\circ\varphi^{-1}.
\]
By the birational superrigidity of $X_4$
\cite{IskovskikhManin} and the identification of its automorphism group
in \eqref{eq-full-fermat-group},
\[
 \operatorname{Bir}(X_4)=\operatorname{Aut}(X_4)=\Gamma.
\]
After identifying $\operatorname{Bir}(X_4)$ with $\Gamma$, the obtained
map is an injective endomorphism of the finite group $\Gamma$, and hence
it is surjective.  This proves the assertion.
\end{proof}

\begin{proof}[Proof of Theorem \ref{thm-sharp-bounds}]
Put $G_{m,r}=(\Z/m)^r$, where $m\geq2$.  For $r\geq4$, Corollary
\ref{cor-kollar-zhuang-uniform} shows that $m^{r-3}$ divides $24$.
If $r\geq7$, this is impossible.  If $r=5$ or $r=6$, it implies that
$m=2$.  If $r=4$, it implies that
\[
 m\in\{2,3,4,6,8,12,24\}.
\]

The case $m=6$ is excluded by Lemma \ref{lem-no-six-four}.  Suppose
that $m\in\{8,12,24\}$ and that
$G_{m,4}\subset\operatorname{Bir}(X)$ for some rationally connected
threefold $X$.  The group $G_{m,4}$ contains a subgroup
$A\simeq(\Z/4)^4$.  By Theorem \ref{thm-classification},
$G_{m,4}$ embeds into
$\operatorname{Bir}(X_4)=\operatorname{Aut}(X_4)
=A\rtimes\mathfrak S_5$.
However,
\[
 |\operatorname{Aut}(X_4)|=4^4\cdot5!=2^{11}\cdot3\cdot5,
\]
and $|G_{m,4}|=m^4$ does not divide this number for
$m\in\{8,12,24\}$, a contradiction.  Thus,
for $r=4$, only $m=2,3,4$ can occur on a rationally connected
threefold.  Corollary \ref{cor-cremona} excludes $m=4$ in
$\Cr_3(\C)$.

Example~\ref{ex-standard-actions} realizes all the cases above.
Part~(1) gives $r\leq3$ for every $m$. Parts~(2) and~(3) give,
respectively, $m=2$, $r\leq6$ and $m=3$, $r\leq4$ on rational
threefolds. Part~(4) gives $m=4$, $r\leq4$ on rationally
connected threefolds.
\end{proof}

\Addresses

\end{document}